\pdfoutput=1

\documentclass[12pt,a4paper]{amsart}

\usepackage[utf8]{inputenc}

\usepackage{lmodern}
\usepackage[T1]{fontenc}
\usepackage[tracking]{microtype} 

\usepackage[style=numams, 
            sorting=nyt, 
            isbn=false,  
            backref=true, 
            backend=biber,
            safeinputenc,
            urldate=ymd,
            maxbibnames=5,
            minbibnames=3,
            maxitems=1]{biblatex}
\usepackage{amsmath}  
\usepackage{amssymb}

\usepackage{mathtools}

\usepackage{dsfont}  

\usepackage[colorlinks, 
            linkcolor=blue, 
            citecolor=blue
           ]{hyperref}

\usepackage[capitalize]{cleveref}

\usepackage[hcentering,
            hscale=0.71,
            vscale=0.73]{geometry}

\swapnumbers

\newtheorem{thm}{Theorem}[section]
\newtheorem{cor}[thm]{Corollary}
\newtheorem{lemm}[thm]{Lemma}
\newtheorem{conjecture}[thm]{Conjecture}

\theoremstyle{definition}

\newtheorem{defi}[thm]{Definition}
\newtheorem{example}[thm]{Example}
\newtheorem{remark}[thm]{Remark}

\renewcommand{\phi}{\varphi}

\renewcommand{\geq}{\geqslant}
\renewcommand{\leq}{\leqslant}

\newcommand{\defemph}[1]{\emph{#1}} 

\newcommand{\reals}{\mathbb{R}}
\newcommand{\compl}{\mathbb{C}}
\newcommand{\quats}{\mathbb{H}}
\newcommand{\nats}{\mathbb{N}}
\newcommand{\ints}{\mathbb{Z}}
\newcommand{\rats}{\mathbb{Q}}

\newcommand{\iso}{\cong}    

\DeclarePairedDelimiter{\card}{\lvert}{\rvert}  
\DeclarePairedDelimiter{\abs}{\lvert}{\rvert}  
\DeclarePairedDelimiter{\norm}{\lVert}{\rVert} 
\DeclarePairedDelimiter{\erz}{\langle}{\rangle}

\newcommand*{\allone}{\mathds{1}}

\DeclareMathOperator{\id}{Id}

\DeclareMathOperator{\tr}{tr}

\DeclareMathOperator{\Z}{Z}            
\DeclareMathOperator{\C}{\mathbf{C}}   
\DeclareMathOperator{\N}{\mathbf{N}}   

\DeclareMathOperator{\GL}{GL}
\newcommand{\Sym}[1]{\mathrm{S}_{#1}}  

\DeclareMathOperator{\Aut}{Aut}
\DeclareMathOperator{\Gal}{Gal}
\DeclareMathOperator{\Units}{\mathbf{U}}

\DeclareMathOperator{\enmo}{End}

\DeclareMathOperator{\Fix}{Fix}
\DeclareMathOperator{\Irr}{Irr}

\DeclareMathOperator{\mat}{\mathbf{M}}

\DeclareMathOperator{\conv}{conv}
\DeclareMathOperator{\orbpol}{P}

\newcommand{\TheTitle}{Equivalence of Lattice Orbit Polytopes}
\newcommand{\TheAuthors}{Frieder Ladisch and Achill Sch\"urmann}

\hypersetup{pdfauthor={\TheAuthors}, 
            pdftitle={\TheTitle},
            }

\begin{document}


\title{\TheTitle}
\author{Frieder Ladisch}
\address{Universität Rostock,
         Institut für Mathematik,
         Ulmenstr.~69, Haus~3,
         18057 Rostock,
         Germany}
\email{frieder.ladisch@uni-rostock.de}
\author{Achill Sch\"urmann}
\email{achill.schuermann@uni-rostock.de}
\subjclass[2010]{Primary 20C10; 
    Secondary 16U60, 20B25, 20C15, 52B20, 90C10}
\keywords{Orbit polytope, core points, group representation, 
    lattice, integer linear programming}
\thanks{The authors were supported by the DFG (Project: SCHU 1503/6-1)}

\dedicatory{Dedicated to Jörg M.\,Wills
on the occasion of his 80th birthday}




\begin{abstract}
  Let $G$ be a finite permutation group acting
  on $\mathbb{R}^d$ by permuting coordinates.
  A \emph{core point} (for $G$) is an 
  integral vector $z\in \mathbb{Z}^d$ such that
  the convex hull of the orbit $Gz$ contains no other
  integral vectors but those in the orbit $Gz$.
  Herr, Rehn and Schürmann considered the question 
  for which groups there are infinitely many core points up
  to \emph{translation equivalence},
  that is, up to translation by vectors fixed by the group.
  In the present paper, we propose a coarser equivalence relation
  for core points
  called \emph{normalizer equivalence}. 
  These equivalence classes often contain infinitely many
  vectors up to translation, for example when
  the group admits an irrational invariant subspace
  or an invariant irreducible subspace occurring with multiplicity
  greater than $1$.
  We also show that the number of core points up to normalizer 
  equivalence is finite if $G$ is a so-called \emph{QI-group}.
  These groups include all transitive permutation groups
  of prime degree.
  We 
  give an example to show how the concept of normalizer equivalence 
  can be used 
  to simplify integer convex optimization problems.  
\end{abstract}

\maketitle


%



\section{Introduction}
Let $G\leq \GL(d,\ints)$ be a finite group.
We consider orbit polytopes 
$\conv(Gz)$ of integral vectors $z\in \ints^d$, 
that is, the convex hull of an orbit of a point
$z$ with integer coordinates.
We call $z$ a \emph{core point} for $G$ when
the vertices are the only integral vectors in the orbit polytope
$\conv(Gz)$.
Core points were introduced in \cite{BoediHerrJoswig13,herrrehnschue13}
in the context of convex integer optimization,
in order to develop new techniques to exploit symmetries.
Core points are relevant to symmetric convex integer optimization,
since  a $G$-symmetric convex set contains an 
integer vector if and only if it contains a core point for $G$.
So when a $G$-invariant convex integer optimization problem has a 
solution,
then there is a core point attaining the optimal value.
As the set of core points is itself $G$-symmetric,
it even suffices to consider only one core point from each $G$-orbit.
In this way, solving a $G$-invariant convex integer optimization problem
can be reduced to a considerably smaller set of integral vectors.
Furthermore, since core points are known to be close to some
$G$-invariant subspace~\cite[Theorem~3.13]{rehn13diss}
\cite[Theorem~9]{herrrehnschue15},
one can for example use them to 
add additional linear or quadratic constraints to a given 
symmetric problem (see, for instance, \cite[Section~7.3]{rehn13diss}).

Previous work on core points mainly considered groups
for which there are only finitely many core points up to
a natural equivalence relation called
\emph{translation equivalence}
\cite{BoediHerrJoswig13,herrrehnschue13,herrrehnschue15}.
For some of these groups, 
even a naive enumeration approach is sufficient
to beat state-of-the-art commercial solvers. 
Moreover, when there are only finitely many core points up to 
translation equivalence,
then one can parametrize the core points of $G$ in a natural way, 
and thereby obtain a beneficial reformulation of a $G$-invariant
problem~\cite{herrrehnschue13}.
This technique was used to solve a previously open problem from 
the MIPLIB 2010~\cite{KochEtAl11}
(see~\cite{herrrehnschue13}).
More elaborate algorithms taking advantage of core points, 
possibly combined with classical techniques from integer optimization,
have yet to be developed.  

In this paper we extend the list of groups for which  
core points can be parametrized.
This is achieved by introducing a new equivalence relation 
for core points,
which is coarser than the translation equivalence previously used. 
It turns out that this new equivalence relation
not only helps to classify core points, 
but also suggests a new way to transform $G$-invariant
convex integer optimization problems in a natural way into possibly
simpler equivalent ones.
Knowing a group $G$ of symmetries, 
elements of its normalizer in 
$\GL(d,\ints)$ 
can be used to transform a 
$G$-invariant convex integer optimization problem 
linearly into an equivalent $G$-invariant problem. 
As we  show in \cref{sec:app} 
for the case of integer linear problem instances, 
the transformed optimization problems are sometimes  
substantially easier to solve.
To apply this technique in general, one needs to know how to find
a good transformation from the normalizer 
which yields an easy-to-solve transformed problem.
While this is easy in some cases as in our examples,
we do not yet understand satisfactorily how to find
good transformations from the normalizer in general.

In the following, we write
\[
   \Fix(G) = \{ v\in \reals^d 
                \mid 
                gv=v \text{ for all } g\in G
             \}
\]
for the fixed space of $G$ in $\reals^d$.
When $z$ is a core point and $t\in \Fix(G) \cap \ints^d$, 
then $z+t$ is another core point.
We call the core points $z$ and $z+t$ \emph{translation equivalent}.
Herr, Rehn and Schürmann~\cite{herrrehnschue15} consider the question
of whether there are finitely  or infinitely many core points up to 
translation equivalence in the case where 
$G$ is a permutation group acting by permuting coordinates.
Their methods can be used to show that there are only finitely many
core points up to translation 
when $\reals^d / \Fix(G)$ has no $G$-invariant subspaces
other than the trivial ones~\cite[Theorem~3.13]{rehn13diss}.
They conjectured that in all other cases,
there are infinitely many core points up to translation.
This has been proved in special cases but is open in general.

In this paper, we consider a coarser equivalence relation, 
where we allow
to multiply core points with invertible integer matrices
$S\in \GL(d,\ints)$ which normalize the subgroup $G$.
Thus two points $z$ and $w$ are called 
\emph{normalizer equivalent}, when
$w= Sz +t$, where $S$ is an element of the 
\emph{normalizer} of $G$ in $\GL(d,\ints)$
(in other words, $S^{-1} G S = G$),
and $t\in \Fix(G)\cap \ints^d$.
In \cref{t:finitecrit}, we will determine 
when these coarser equivalence classes contain infinitely many 
points up to translation equivalence, in terms of the 
decomposition into irreducible invariant subspaces.
For example, if $\reals^d$ has an \emph{irrational}
invariant subspace $U\leq \reals^d$
(that is, a subspace $\{0\}\neq U\leq \reals^d$ such that
 $U \cap \ints^d = \{0\}$),
then each integer point $z$ with nonzero projection onto $U$
is normalizer equivalent to infinitely many points,
which are not translation equivalent.
This yields another proof of the result of
Herr, Rehn and Schürmann~\cite[Theorem~32]{herrrehnschue15}
that there are infinitely many core points up to translation,
when there is an irrational invariant subspace.

In \cref{t:QI_finitelycorep}, we prove the following:
Suppose that $G\leq\Sym{d}$ is a transitive permutation group acting on
$\reals^d$ by permuting coordinates.
Suppose that $\Fix(G)^{\perp}$ contains no rational 
$G$-invariant subspace other than $\{0\}$ and
$\Fix(G)^{\perp}$ itself.
(A subspace of $ \reals^d$ is rational if 
 it has a basis contained in $\rats^d$.)
Such a group $G$ is called a \emph{QI-group}.
Then there are only finitely many core points up to normalizer
equivalence.

For example, this is the case when $d=p$ is a prime number
(and $G\leq \Sym{p}$ is transitive).
In the case that the group is not 
$2$-homogeneous, there are infinitely many
core points up to translation, 
but these can be obtained from finitely many by multiplying
with invertible integer matrices from the normalizer.

The paper is organized as follows.
In \cref{sec:defs},
we introduce different equivalence relations
for core points and make some elementary observations.
\Cref{sec:orders} collects some elementary 
properties of orders in semisimple algebras.
In \cref{sec:finiteequiv},
we determine when the normalizer equivalence classes
contain infinitely many points up to translation equivalence.
In \cref{sec:qi-groups},
we prove the aforementioned result on QI-groups.
\Cref{sec:finiteequiv,sec:qi-groups}
can mostly be read independently from one another.
Finally, in \cref{sec:app} we give an example to show how 
normalizer equivalence can be applied 
to integer convex optimization problems with suitable symmetries.

\section{Equivalence for core points}
\label{sec:defs}
  Let $V$ be a finite-dimensional vector space 
  over the real numbers $\reals$ 
  and $G$ a finite group acting linearly on $V$.

  \begin{defi}
    An \defemph{orbit polytope} (for $G$)
    is the convex hull of the $G$-orbit of a point $v\in V$.
    It is denoted by 
    \[ \orbpol(G,v) = \conv\{gv\mid g\in G\}.\]
  \end{defi}

  Let $\Lambda\subseteq V$ be a full 
  $\ints$-lattice in $V$, that is, 
  the $\ints$-span of an
  $\reals$-basis of $V$.
  Assume that $G$ maps $\Lambda$ onto itself.   

  \begin{defi}\cite{herrrehnschue13}
    An element $z\in \Lambda$ is called a
    \defemph{core point} (for $G$ and $\Lambda$) if
    the only lattice points in $\orbpol(G,z)$ are its vertices,
    that is, the elements in the orbit $Gz$.
    In other words, $z$ is a core point if
    \[ \orbpol(G,z)\cap \Lambda = Gz.\]
  \end{defi}

    The barycenter
    \[ \frac{1}{\card{G}}
       \sum_{g\in G} gv \in \orbpol(G,v)
    \]
    is always fixed by $G$. 
    If $\Fix_V(G)$, the set of vectors in $V$
    fixed by all $g\in G$, consists only of $0$,
    then the barycenter of each orbit polytope is
    the zero vector.
    In this case, only the zero vector is a core point
    \cite[Remark~3.7, Lemma~3.8]{rehn13diss}.
    
    More generally, the map
    \[ e_1 = 
       \frac{1}{\card{G}}
         \sum_{g\in G} g
    \]
    gives the projection from $V$ onto the fixed space 
    $\Fix_V(G)$~\cite[\S 2.6, Theorem~8]{SerreLRFG},
    and thus yields a decomposition
    $V = \Fix_V(G) \oplus \ker(e_1)$ into $G$-invariant
    subspaces.
    If this decomposition restricts to a decomposition
    of the lattice,
    $\Lambda = e_1\Lambda \oplus (\ker(e_1)\cap \Lambda)$,
    then $e_1z\in \Lambda \cap P(G,z)$ for any  $z\in \Lambda$,
    and so $z$ can only be a core point for $G$
    when $z$ is itself in the fixed space.
    But in general, we do not have such a decomposition,
    since the projection 
    $e_1\Lambda$ may not be contained in $\Lambda$.
     
    An important class of examples where 
    $e_1\Lambda \not\subseteq \Lambda$
    is transitive permutation groups $G \leq \Sym{d}$,
    acting on $V=\reals^d$ by permuting coordinates, 
    and where $\Lambda = \ints^d$.
    The fixed space consists of the vectors with all entries
    equal and is thus generated by the all ones vector
    $\allone:=(1,1,\dotsc,1)^t$.
    For $v=(v_1,\dotsc, v_d)^t$ we have
    $e_1v = (\sum_i v_i)/d \cdot \allone$.
    In particular,  we see that
    $e_1\Lambda$ contains all integer multiples
    of $(1/d)\allone$.
    We can think of $\Lambda$ as being partitioned into 
    \emph{layers}, where a layer consists 
    of all $z\in \Lambda$ with $z^t\allone = k$
    (equivalently, $e_1z= (k/d)\allone$),
    for a fixed integer~$k$.
    
    Returning to general groups of integer matrices, 
    we claim that for each
    $v\in e_1\Lambda $,
    there are core points $z$ with
    $e_1z=v$.
    Namely, among all $z\in \Lambda$
    with $e_1z=v$, 
    there are elements such that 
    $\sum_g\norm{gz}^2$ is minimal,
    and these are core points.

  If $z$ is a core point and $b\in \Fix_{\Lambda}(G)$,
  then $z+b$ is a core point, too,
  because $\orbpol(G,z+b) = \orbpol(G,z) + b$.
  Such core points should be considered as 
  \emph{equivalent}.
  This viewpoint was adopted by 
  Herr, Rehn and Schürmann~\cite{herrrehnschue13,herrrehnschue15}.
  In the present paper, we consider a coarser equivalence relation.
  We write $\GL(\Lambda)$ for the invertible 
  $\ints$-linear maps $\Lambda \to \Lambda$.
  Since $\Lambda$ contains a basis of $V$, we may view
  $\GL(\Lambda)$ as a subgroup of $\GL(V)$.
  (If $V= \reals^d$ and $\Lambda= \ints^d$, then we can identify
  $\GL(\Lambda)$  with $\GL(d,\ints)$,
  the set of matrices over $\ints$ with determinant $\pm 1$.)
  
  By assumption, $G$ is a subgroup of $\GL(\Lambda)$.
  We use the following notation from group theory:
  The \defemph{normalizer} of a finite group $G$ in 
  $\GL(\Lambda)$ is the set
  \[ \N_{\GL(\Lambda)}(G) :=
     \{ S\in \GL(\Lambda) \mid
        \forall g\in G\colon S^{-1} g S \in G
     \}.
  \]  
  The \defemph{centralizer} of $G$ in $\GL(\Lambda)$
  is the set 
  \[ \C_{\GL(\Lambda)}(G) := 
       \{ S\in \GL(\Lambda) \mid
          \forall g\in G\colon
          S^{-1} g S = g
       \}.
  \]

  \begin{defi}\label{df:equiv}
    Two points $z$ and $w$ are called 
    \defemph{normalizer equivalent}
    if there is a point $b\in \Fix_{\Lambda}(G)$ and
    an element $S$ in the normalizer $\N_{\GL(\Lambda)}(G)$ 
    of $G$ in $\GL(\Lambda)$ such that 
    $w = Sz + b$.
    The points are called \defemph{centralizer equivalent}
    if $w= Sz+b$ with $S\in \C_{\GL(\Lambda)}(G)$
    and $b\in \Fix_{\Lambda}(G)$.
    Finally, we call two points $z$ and $w$ 
    \defemph{translation equivalent}
    when $w-z \in \Fix_{\Lambda}(G)$.
  \end{defi}  

  Since $\C_{\GL(\Lambda)}(G)\subseteq \N_{\GL(\Lambda)}(G)$,
  each normalizer equivalence class is a union 
  of centralizer equivalence classes,
  and obviously each centralizer equivalence class
  is a union of translation equivalence classes.
  The definition is motivated by the following simple observation:

  \begin{lemm}\label{l:equiv_core_pts}
    If 
    \[ w = Sz +b
       \quad \text{with}\quad 
       S \in \N_{\GL(\Lambda)}(G)
       \quad \text{and} \quad
                  b \in \Fix_{\Lambda}(G), 
    \]
    then
    $x \mapsto Sx +b$ defines a bijection between
    \[ \orbpol(G,z) \cap \Lambda
       \quad \text{and} \quad
       \orbpol(G,w) \cap \Lambda.
    \]
    In particular, $z$ is a core point for $G$ 
    if and only if $w$ is a core point for $G$.
  \end{lemm}  

  \begin{proof}
    The affine bijection 
    $ x \mapsto Sx + b $ maps the orbit polytope
    $\orbpol(G,z)$ to another polytope.
    The vertex $gz$ is mapped to the vertex
    \[
      S gz + b = (SgS^{-1}) Sz + b = hSz +b = h(Sz+b) = hw,
    \]
    where $h= SgS^{-1}\in G$ (since $S$ normalizes $G$). 
    The second to last equality follows as $b$ is fixed by $G$.
    As $SgS^{-1}$ runs through $G$ as $g$ does, it follows that
    $x\mapsto Sx+b$ maps the orbit $Gz$ to the 
    orbit $Gw$ and thus maps the orbit polytope
    $\orbpol(G,z)$ to the orbit polytope $\orbpol(G,w)$.
    Since $x\mapsto Sx+b$ also maps $\Lambda$ onto itself,
    the result follows.    
  \end{proof}

    Notice that a point $w$ is equivalent to $z=0$
    (for any of the equivalence relations
    in \cref{df:equiv})
    if and only if 
    $w = S\cdot 0 +b = b\in \Fix_{\Lambda}(G)$. 
    Any $w\in \Fix_{\Lambda}(G)$ is a core point.
    We call these points the \defemph{trivial core points}.
    
    In the important example of transitive permutation groups,
    the fixed space is one-dimensional.
    More generally, when $V$ is spanned linearly by some 
    orbit~$Gz$, then $\Fix_V(G)$ is spanned by $e_1z$ and thus
    $\dim(\Fix_V(G)) \leq 1$.
    
  \begin{remark}\label{rem:nrmtrans}
    Suppose that $\Fix_V(G)$ has dimension~$1$.
    Then there is at most one $w\in \N_{\GL(\Lambda)}(G)z$ 
    such that $w\neq z$ 
    and $w$ is translation equivalent to $z$.
  \end{remark}
      
  \begin{proof}
    The elements of $\N_{\GL(\Lambda)}(G)$ map
    $\Fix_{\Lambda}(G)$ onto itself and thus act on
    $\Fix_V(G)$ as $\pm 1$.    
    Let $S\in \N_{\GL(\Lambda)}(G)$.
    Suppose $w=Sz$ and $z$ are translation equivalent,
    so that $Sz-z=b\in \Fix_{\Lambda}(G)$.
    Then $b=e_1b =e_1Sz-e_1z = Se_1 z - e_1 z  
            = \pm e_1 z -e_1 z$
    and thus either $Sz=z$ or $Sz = z - 2e_1 z$.
    (The latter case can only occur when
    $2e_1z\in \Lambda$.)    
 \end{proof}   
    
    In particular, if the orbit
    $\N_{\GL(\Lambda)}(G)z$ is infinite,
    then the normalizer equivalence class of 
    a nontrivial core point~$z$
    contains infinitely many translation equivalence
    classes. 
     
    Also notice that when $z$ is a nontrivial core point,
    then $e_1z$ must not be a lattice point.

    Herr, Rehn, and 
  Schürmann~\cite{herrrehnschue15,Herr13_Diss,rehn13diss}
  considered the question of whether the set of core points
  up to translation is finite or infinite
  (in the case where $G$ acts by permuting coordinates).
  We might ask the same question about core points up to
  normalizer equivalence as defined here.
  Also, it is of interest whether our bigger equivalence classes
  contain finitely or infinitely many points
  up to translation.

\begin{example}
   Let $G= \Sym{d}$, the symmetric group on $d$ elements,
   acting on $\reals^d$ by permuting coordinates,
   and $\Lambda=\ints^d$.
   We identify $G$ with the group of all permutation matrices.
   For this group, Bödi, Herr and
   Joswig~\cite{BoediHerrJoswig13} have shown that every
   core point is translation equivalent to a vector
   with all entries $0$ or $1$.
   (Conversely, these vectors are obviously core points.)
   One can show that the normalizer of the group $G$ of
   all permutation matrices in $\GL(d,\ints)$ is generated
   by $-I$ and the group $G$ itself.
   As $G$ is transitive on the subsets of $\{1,\dotsc,d\}$
   of size $k$,
   all $0/1$-vectors with fixed number $k$ of $1$'s are 
   normalizer equivalent.
   A vector $z$ with $k$ ones and $d-k$ zeros is also
   normalizer equivalent to the vector
   $-z + \allone$ with $d-k$ ones and $k$ zeros.
   Thus up to normalizer equivalence, there are only
   $\lfloor d/2 \rfloor +1$ core points.
\end{example}

\begin{example}\label{ex:cycgen}
   Let $G=C_d = \erz{(1,2,\dotsc,d)}$ be a cyclic group,
   again identified with a matrix group
   which acts on $\reals^d$ 
   by permuting the coordinates cyclically.
   For $d=4$ we have a finite normalizer
   (as we will see in \cref{sec:finiteequiv}) but infinitely
   many core points up to normalizer or translation equivalence:
   for example, all the points 
   $(1+m,-m,m,-m)^t$, $m\in \ints$, 
   are core points for $C_4$~\cite[Example~26]{herrrehnschue15}.
   
   If $d=p$ is prime, then we will see that there are only 
   finitely many core points up to normalizer equivalence,
   but for $p\geq 5$ the normalizer is infinite and there 
   are infinitely many core points up to translation equivalence.
   (See \cref{ex:c5} below.)
   
   For $d=8$ (say), 
   the normalizer is infinite \emph{and} there are infinitely many 
   core points up to normalizer equivalence.
   Namely, let $b_1\in \reals^8$ be the first standard basis vector
   and let $v\in \reals^8$ be the vector with entries
   alternating between $1$ and $-1$.
   Then the points
   $b_1 + m v$ for $m\in \ints$ are 
   core points~\cite[Theorem~30]{herrrehnschue15}
   (the construction principle here is the same as above in
   the case $d=4$).
   The circulant 
   $8\times 8$-matrix $S$ with
   first row $(2,1,0,-1,-1,-1,0,1)$
   is contained in the centralizer of $G$ and has infinite order.
   Since $S$ is symmetric and $Sv=v$, we have
   $v^t S^k b_1  = v^t b_1 =1$ for all $k\in \ints$ and thus 
   the vectors $S^kb_1 + mv $ are all different for different
   pairs $(k,m)\in \ints^2$.
   And since we also have $S\allone = \allone$,
   where $\allone = (1,1,\dotsc,1)^t$ spans the fixed space,
   we also see that different vectors of the form
   $S^kb_1 + mv$ can not be translation equivalent.   
   Finally, one can show that  
   the subgroup generated by $S$ has finite index in the 
   normalizer $\N_{\GL(8,\ints)}(C_8)$.
   Thus at most finitely many of the points 
   $b_1 + mv$ can be normalizer equivalent to each other.
\end{example}

It is sometimes easier to work with the 
centralizer $\C_{\GL(\Lambda)}(G)$ instead of the normalizer
$\N_{\GL(\Lambda)}(G)$,
which yields a slightly finer equivalence relation.
By the following simple observation, the 
$\C_{\GL(\Lambda)}(G)$-equivalence classes can not be much 
smaller than the
$\N_{\GL(\Lambda)}(G)$-equivalence classes:

\begin{lemm}\label{l:normcentfin}
   $\card{\N_{\GL(\Lambda)}(G):\C_{\GL(\Lambda)}(G)}$ is finite.
\end{lemm}

\begin{proof}
     The factor group $\N_{\GL(\Lambda)}(G)/\C_{\GL(\Lambda)}(G)$
     is isomorphic to a subgroup of 
     $\Aut(G)$~\cite[Corollary~X.19]{IsaacsFGT}, 
     and $\Aut(G)$ is finite,
     since $G$ itself is finite by assumption.
\end{proof}

\section{Preliminaries on orders}
\label{sec:orders}
In this section, we collect some simple properties
of \emph{orders} in semisimple algebras over $\rats$.
Orders are relevant for us since the centralizer
$\C_{\GL(\Lambda)}(G)$ can be identified with the 
\emph{unit group} of such an order, as we explain below.

Recall the following definition~\cite{reinerMO}:
  Let $A$ be a finite-dimensional algebra over $\rats$
  (associative, with one).  
  An \defemph{order} (or \defemph{$\ints$-order})
  in $A$ is a subring $R \subset A$
  which is finitely
  generated as a $\ints$-module and 
  contains a $\rats$-basis of $A$. 
  (Here, ``subring'' means in particular that $R$ and
  $A$ have the same multiplicative identity.)   
  In other words, an order is a full $\ints$-lattice
  in $A$ which is at the same time a subring of $A$.  
  
For the moment, assume
that $W$ is a finite-dimensional vector space over the rational
numbers $\rats$,
and let $\Lambda$ be a full $\ints$-lattice in $W,$
that is, the $\ints$-span of a $\rats$-basis of $W,$
and $G$ a finite subgroup of $\GL(\Lambda)$.
(In the situation of \cref{sec:defs},
we can take for $W$ the $\rats$-linear span
of $\Lambda$.)
Let
$A:=\enmo_{\rats G} (W)$ be the ring of $\rats G$-module
endomorphisms of $W$, that is, the set of linear maps
$\alpha\colon W\to W$ such that
$\alpha(gv) = g\alpha(v)$ for all $v\in W$ and $g\in G$.
This is just the centralizer of $G$ in the ring of all
$\rats$-linear endomorphisms of $W$.

We claim that
  \[ R:= 
         \{ \alpha \in A
            \mid
            \alpha(\Lambda)\subseteq \Lambda
         \}
  \]
is an order in $A$.
Namely, choose a $\ints$-basis of $\Lambda$.
This basis is also a $\rats$-basis of $W$.
By identifying linear maps with matrices with respect to the
chosen basis,
$A$ gets identified with the centralizer of $G$
in the set of \emph{all} $d\times d$ matrices over $\rats$,
and $R$ gets identified with the centralizer of $G$ in
the set of 
  $d\times d$ matrices with entries in $\ints$.
It follows that $R$ is finitely generated as a $\ints$-module,
and for every $\alpha\in A$ there
is an $m\in \ints$ such that $m\alpha \in R$.
Thus $R$ is an order of $A$.
(Also, $R\iso \enmo_{\ints G}(\Lambda)$ naturally.)
  
Moreover, the centralizer
$\C_{\GL(\Lambda)}(G)$ is exactly the set of invertible
elements of $R$, that is, the unit group 
$\Units(R)$ of $R$.
  For this reason, it is somewhat easier to work
  with $\C_{\GL(\Lambda)}(G)$ instead of
  the normalizer $\N_{\GL(\Lambda)}(G)$.
The unit group $\Units(R)$ of an order $R$ is a finitely
generated (even finitely presented) group~\cite[Section~3]{Kleinert94}.
Finding explicit generators of $\Units(R)$
(and relations between them)
is in general a difficult task, but there do exist algorithms
for this purpose~\cite{BraunCNS15}.
The situation is somewhat better when
$R$ is commutative, for example when
$R \iso \ints A$, where $A$ is a finite 
abelian group~\cite{FaccindeGraafPlesken13}.
Moreover, it is quite easy to give generators of a 
subgroup of $\Units(\ints A)$ which has finite index
in $\Units(\ints A)$~\cite{Hoechsmann92,MarciniakSehgal05}.
  
  We now collect some general elementary facts about orders
  that we need.
  (For a comprehensive treatment of orders (not only over
  $\ints$), we refer the reader to Reiner's book on 
  maximal orders~\cite{reinerMO}.
  For unit groups of orders,
  see the survey article by Kleinert~\cite{Kleinert94}.)

  \begin{lemm}\label{l:intersectorders}
    Let $R_1$ and $R_2$ be two orders in the
    $\rats$-algebra $A$.
    Then $R_1 \cap R_2$ is also an order
    in $A$.
  \end{lemm}  

  \begin{proof}
      Clearly, $R_1\cap R_2$ is a subring.
      
      Since $R_2$ is finitely generated over $\ints$ and 
      $\rats R_1=A$,
      there is a non-zero integer $m \in \ints$ with 
      $mR_2 \subseteq R_1$.
      Thus $mR_2 \subseteq R_1 \cap R_2$.
      Since $mR_2$ contains a $\rats$-basis of $A$,
      it follows that $R_1\cap R_2$ contains such a basis.
      As a submodule of a finitely generated $\ints$-module,
      $R_1 \cap R_2$ is again finitely generated.
      Thus $R_1\cap R_2$ is an order of $A$.  
  \end{proof}

  \begin{lemm}\label{l:suborder}
    Let $R_1$ and $R_2$ be orders in the 
    $\rats$-algebra $A$ with 
    $R_1 \subseteq R_2$.
    Then 
    $\card{\Units(R_2) : \Units(R_1)}$ is finite.
  \end{lemm}

  \begin{proof}    
    There exists a non-zero integer $m$ such that
    $m R_2 \subseteq R_1$.
    Suppose that $u$, $v\in \Units(R_2)$ are such that
    $u-v \in mR_2$.
    Then $u \in v + mR_2$ and thus
    $uv^{-1} \in 1 + mR_2 \subseteq R_1$.
    Similarly, $vu^{-1} \in 1 + mR_2 \subseteq R_1$.
    Thus $uv^{-1} \in \Units(R_1)$.
    This shows 
    $\card{\Units(R_2): \Units(R_1)} 
      \leq \card{R_2 : mR_2} < \infty$,
      as claimed.
  \end{proof}

  \begin{cor}\label{c:changeorder}
    Let $R_1$ and $R_2$ be two orders in the
    $\rats$-algebra $A$.
    Then $\Units(R_1)$ is finite if and only if
    $\Units(R_2)$ is finite.    
  \end{cor}

  \begin{proof}
    By \cref{l:intersectorders}, 
    $R_1 \cap R_2$ is an order.
    By \cref{l:suborder}, the index 
    $\card{ \Units(R_i) : \Units(R_1\cap R_2) }$
    is finite for $i=1$, $2$.
    The result follows.    
  \end{proof}
 
\section{Finiteness of equivalence classes}
\label{sec:finiteequiv}
  In this section we determine for which groups $G$
  the normalizer equivalence classes are finite or not.
  We use the notation introduced in \cref{sec:defs}.
  Thus $G$ is a finite group acting on the
  finite-dimensional, real vector space $V$,
  and $\Lambda\subset V$ is a full $\ints$-lattice in $V$
  which is stabilized by $G$.
  A subspace $U \leq V$ is called 
  \defemph{$\Lambda$-rational} 
  if $U\cap \Lambda$ contains a basis of $U$, and
  \defemph{$\Lambda$-irrational}
  if $U\cap \Lambda = \{0\}$.
  If $U$ is an irreducible $\reals G$-submodule, 
  then $U$ is either $\Lambda$-rational or 
  $\Lambda$-irrational.
  
  \begin{thm}\label{t:finitecrit}
    Let 
    \[ V = 
       U_1 \oplus \dotsb \oplus U_r
    \]
    be a decomposition of $V$ into irreducible
    $\reals G$-subspaces.    
    Then 
    $\N_{\GL(\Lambda)}(G)$ has finite order if and only if
    all the $U_i$'s are $\Lambda$-rational and pairwise non-isomorphic.
  \end{thm}

  The proof of \cref{t:finitecrit} 
  involves some non-trivial representation and number theory.
  By \cref{l:normcentfin}, 
  the normalizer $\N_{\GL(\Lambda)}(G)$ is finite if and only if
  the centralizer $\C_{\GL(\Lambda)}(G)$ is finite.
  As remarked earlier, 
  the centralizer can naturally be identified with the set of units 
   of the ring $\enmo_{\ints G}(\Lambda)$,
  and $\enmo_{\ints G}(\Lambda)$ is an order in the 
  $\rats$-algebra
  $\enmo_{\rats G}( \rats \Lambda) $,
  where $\rats \Lambda$ denotes the $\rats$-linear span 
  of $\Lambda$. 
  For this reason, it is more convenient to work with the 
  $\rats$-vector space $W:= \rats \Lambda$.
  (We get back our $V$ from $W$ by scalar extension,
  that is, $V\iso \reals \otimes_{\rats} W$.)

  Fix a decomposition of $W= \rats \Lambda$ into 
  simple modules:
  \[ W \iso m_1 S_1 \oplus \dotsb \oplus m_r S_r,
     \quad m_i \in \nats,
  \]
  where we assume that $S_i \not\iso S_j$ for
  $i\neq j$.
  Set $D_i := \enmo_{\rats G} (S_i)$, which is by 
  Schur's lemma~\cite[(3.6)]{lamFC}
  a division ring, and finite-dimensional over $\rats$.

  \begin{lemm}\label{l:endo_decom}
    With the above notation, we have
    \[ \enmo_{\rats G}(W)
       \iso
       \mat_{m_1}(D_1) \times \dotsb \times 
       \mat_{m_r}(D_r),
    \]  
    where $\mat_m(D)$ denotes the ring of $m\times m$ matrices
    with entries in $D$.
    If $R_i$ is an order in $D_i$ for each $i$, then
    \[ R:=
       \mat_{m_1}(R_1) \times \dotsb \times
       \mat_{m_r}(R_r)
    \]
    is an order in $\enmo_{\rats G}(W)$.
  \end{lemm}

  \begin{proof}
    The first assertion is a standard observation, 
    used, for example, in one proof of the 
    Wedderburn-Artin structure theorem for semisimple 
    rings~\cite[Thm.~3.5 and proof]{lamFC}.
    The assertion on orders is then easy. 
  \end{proof}

  In particular, the group of units of $R$ is then isomorphic 
  to
  the direct product of groups of the form
  $\GL(m_i, R_i)$.
  To prove \cref{t:finitecrit}, 
  in view of \cref{c:changeorder},
  it suffices to determine when all these
  groups are finite.  
  The following is a first step toward the proof of 
  the theorem:

  \begin{cor}\label{c:multfree}
    If some $m_i > 1$, then
    $\Units(R)$
    (and thus $\N_{\GL(\Lambda)}(G)$)
    is infinite.
  \end{cor}

  \begin{proof}
     $\Units(R)$ contains a subgroup isomorphic to
     $\GL(m_i, R_i)$, which contains the group 
     $\GL(m_i, \ints)$.
     This group is infinite if $m_i>1$.
  \end{proof}
  
  To continue with the proof of 
  \cref{t:finitecrit},
  we have to look at the units of an order $R_i$ in $D_i$.
  We will need extension of scalars
  for algebras over a field via tensor products, as explained 
  in \cite[Chapter~3]{FarbDennis93}. 
  Thus for a $\rats$-algebra~$A$, 
  we get an $\reals$-algebra denoted by $\reals \otimes_{\rats} A$.
  We use the following theorem of Käte Hey
  which can be seen as a generalization of Dirichlet's unit theorem:

  \begin{thm}\cite[Theorem~1]{Kleinert94}
    \label{t:kaetehey}
    Let $D$ be a finite-dimensional division algebra
    over $\rats$,
    and let $R$ be an order of $D$
    with unit group $\Units(R)$.
    Set
    \[  S = \{ d\in \reals \otimes_{\rats} D
               \mid
               (\det d)^2 = 1
            \}.
    \]
    Then $S/\Units(R)$ is compact.
    (Here $\det d$ refers to the action of $d$ 
    as linear operator on $\reals \otimes_{\rats} D$.
    One can also use the reduced norm, of course.)
  \end{thm}

  From this, we can derive the following result
  (probably well known):

  \begin{lemm}\label{l:genquats}
    Let $D$ be a finite-dimensional division algebra over $\rats$
    and $R$ an order of $D$.
    Then $\card{\Units(R)}< \infty$ if and only if\/
    $\reals \otimes_{\rats} D$ is a division ring.
  \end{lemm}

  \begin{proof}
    Suppose $D_{\reals}:=\reals \otimes_{\rats} D$ is a division ring.
    By Frobenius's theorem~\cite[Theorem~3.20]{FarbDennis93}, we have
    $D_{\reals} \iso \reals$, $\compl$, or $\quats$.
    In each case, one checks that the set $S$
    defined in \cref{t:kaetehey} is compact.
    Thus the discrete group $\Units(R)\subseteq S$ must be finite.
    (Notice that we did not use
    \cref{t:kaetehey} here---only
    that $\Units(R)\subseteq S$.)
    
    Conversely, suppose that $D_{\reals}$ is not a division ring.
    Then there is some non-trivial idempotent 
    $e\in D_{\reals}$,
    that is, $e^2=e$, but $e \neq 0$, $1$.
    (This follows since $D_{\reals}$ is semisimple.)
    Set $f=1-e$.
    Then for $\lambda$, $\mu\in \reals$,
    we have $\det(\lambda e + \mu f) = \lambda^{k_1}\mu^{k_2}$
    with $k_1 = \dim (D_{\reals}e)$
    and $k_2 = \dim (D_{\reals}f)$.
    In particular, for every 
    $\lambda \neq 0$ there is some $\mu$ such that
    $\lambda e + \mu f \in S$.
    This means that $S$ is unbounded, and thus not compact.
    It follows from \cref{t:kaetehey}
    that $\Units(R)$ can not be finite.
  \end{proof}

  \begin{proof}[Proof of \cref{t:finitecrit}]
    First, assume that we are given a decomposition
    $V = U_1 \oplus \dotsb \oplus U_r$ as in the theorem.
    Then $S_i := U_i \cap \rats\Lambda$ contains a basis
    of $U_i$ and thus is non-zero
    and necessarily simple as a $\rats G$-module.
    Thus
    \begin{align*}
      W = V\cap \rats\Lambda
        &= S_1 
           \oplus \dotsb \oplus
           S_r 
    \end{align*}
    is a decomposition of $W$ into simple
    $\rats G$-modules, which are pairwise 
    non-isomorphic.
    It follows that
    \[ \enmo_{\rats}(W) \iso
       D_1 \times \dotsm \times D_r,
    \]
    where $D_i = \enmo_{\rats G}(S_i)$.
    Since $\reals \otimes_{\rats} D_i 
    \iso \enmo_{\reals G}(U_i)$ is a division ring, too,
    it follows that the orders of each $D_i$ have a finite 
    unit group, by \cref{l:genquats}.
    Thus $\C_{\GL(\Lambda)}(G)$ is finite.
    
    Conversely, assume that $\N_{\GL(\Lambda)}(G)$ is finite.
    It follows from \cref{c:multfree}
    that $m_i=1$ for all $i$
    (in the notation introduced before \cref{l:endo_decom}).
    Thus $W$ has a decomposition into simple summands
    which are pairwise non-isomorphic:
    \[ W = S_1 \oplus \dotsb \oplus S_r.
    \]
    Let $D_i = \enmo_{\reals G}(S_i)$.
    Then \cref{l:genquats} yields that
    $\reals \otimes_{\rats} D_i $ is a division ring, too.
    Since 
    $\reals \otimes_{\rats} D_i \iso \enmo_{\reals G}(\reals S_i)$,
    it follows that $U_i := \reals S_i$ is simple.
    (Otherwise, the projection to a nontrivial invariant
     submodule would be a zero-divisor in
     $\enmo_{\reals G}(U_i)$.)
    For $i\neq j$, we have $U_i \not\iso U_j$
    by the Noether-Deuring theorem~\cite[Theorem~19.25]{lamFC}.
    Thus $V$ has a decomposition 
    $V = U_1 \oplus \dotsb \oplus U_r$ as required.
  \end{proof}
  
  \begin{remark}
      Let $z\in V$ be an element such that
      the orbit $Gz$ linearly spans $V.$
      Then the normalizer equivalence class of $z$ contains 
      infinitely many translation equivalence classes
      if (and only if)
      $\N_{\GL(\Lambda)}(G)$ has infinite order.
  \end{remark}

  \begin{proof} 
      The ``only if'' part is clear,
      so assume that $\N_{\GL(\Lambda)}(G)$ has infinite order.
      By \cref{rem:nrmtrans}, it suffices to show that
      the orbit $\N_{\GL(\Lambda)}(G)z$ has infinite size.
      By \cref{l:normcentfin}, 
      the centralizer $\C_{\GL(\Lambda)}(G)$ has also infinite order.
      If $cz=z$ for $c\in \C_{\GL(\Lambda)}(G)$,
      then $cgz = gcz = gz$ for all $g\in G$ and thus
      $c=1$.
      Thus
      \[ 
         \infty 
         =\card{\C_{\GL(\Lambda)}(G)}
         = \card{\C_{\GL(\Lambda)}(G)z}
         \leq \card{\N_{\GL(\Lambda)}(G)z}. 
      \]    
  \end{proof}
    
    So when $\N_{\GL(\Lambda)}(G)$ is infinite,
    only elements contained in proper invariant subspaces can have
    finite orbits under the normalizer.
    (Notice that the linear span of an orbit $Gz$ is always 
    a $G$-invariant subspace  of $V$.)
    If $G$ is a transitive permutation group acting
    on the coordinates, then there are always points $z$
    such that the orbit $Gz$ spans the ambient 
    space---for example, $z=(1,0,\dotsc,0)^t$.

    When $V$ has 
    an irrational invariant subspace, then 
    $\N_{\GL(\Lambda)}(G)$ is infinite, by \cref{t:finitecrit}.
    Thus if $z$ is a core point for $G$ such that its orbit spans 
    the ambient space, then there are infinitely 
    many core points, even up to translation.
    This was first proved by 
    Rehn~\cite{rehn13diss,herrrehnschue15}
    for permutation groups.
    
    Another consequence of \cref{t:finitecrit} and the remark above
    is that there are infinitely many core points
    for transitive permutations groups $G$ acting on $V=\reals^d$
    such that $V$ is not multiplicity-free (as an $\reals G$-module).

  \begin{example}
    Consider the regular representation of a group $G$, 
    that is, $G$ acts on $\rats G$ by left multiplication,
    so it permutes the canonical basis $G$.
    As a lattice, we choose the group ring
    $\ints G$, the vectors with integer coordinates.
    Then $\enmo_{\ints G}(\ints G) \iso \ints G$.
    Units of group rings are a much studied problem.
    A theorem of Higman says that 
    $\Units(\ints G)$ is finite if and only if
    $G$ is abelian of exponent $1$, $2$, $3$, $4$ or $6$,
    or $G\iso Q_8\times E$ with $E^2 = \{1\}$.
    This can also be derived from 
    \cref{t:finitecrit}.
    
    In \cref{ex:cycgen}, 
    we described some core points in the cases 
    $G=C_4$ and $C_8$.
    In the case of $C_8$, the decomposition 
    of $\rats C_8$ into simple modules is given by
    \[ \rats C_8 
       \iso \rats \oplus \rats \oplus \rats[i] 
        \oplus \rats[ e^{2\pi i/8}].
    \]
    Over $\reals$, the last summand decomposes into 
    two invariant, irrational subspaces of dimension~$2$.
    The normalizer of $C_8$ is infinite because of this last
    summand. 
    Of course, any $z$ contained in the sum of the first 
    three summands has only a finite orbit under the normalizer,
    for example $z=(1,0,0,0,1,0,0,0)^t$.
    
    When $p$ is prime and $p\geq 5$, then 
    $\Units(\ints C_p)$ is infinite,
    but there are only finitely many core points up to 
    normalizer equivalence
    in $\ints C_p$, by \cref{t:QI_finitelycorep} below.
  \end{example}

\section{Rationally irreducible}
\label{sec:qi-groups}
Suppose that $\Lambda= \ints^d $, and assume that 
$G$ acts on $\reals^d$ by matrices in $\GL(d,\ints)$.
A subspace $U \leq \reals^d$ is called 
  \defemph{irrational} if 
  $U\cap \rats^d = \{0\}$ and \defemph{rational} if
  $U$ has a basis contained in $\rats^d$.
  If $U$ is an irreducible $\reals G$-submodule, 
  then $U$ is either rational or irrational.

In this section, we consider permutation groups
acting on $\reals^d$ by permuting coordinates.
(We conjecture that a version of the main result remains true  
more generally for finite matrix groups $G\leq \GL(d,\ints)$,
but we are not able to prove it yet.
One problem is that we can not extend
\cref{l:decomp_ratirr} below  to 
 this more general setting.)
 
Since permutation matrices are orthogonal, 
it follows that the orthogonal complement $U^{\perp}$
of any $G$-invariant subspace is itself $G$-invariant.
Following Dixon~\cite{dixon05}, 
we call a transitive permutation group $G$ a \defemph{QI-group},
when $\Fix(G)^{\perp}$ does not contain any rational 
$G$-invariant subspace other than $\{0\}$ and\/
  $\Fix(G)^{\perp}$ itself.
Notice that $\Fix(G)^{\perp}$ contains no non-trivial rational 
invariant subspaces
if and only if $\Fix(G)^{\perp}\cap \rats^d$ 
contains no proper $G$-invariant subspace 
other than $\{0\}$.
In algebraic language, this means that
$\Fix(G)^{\perp} \cap \rats^d$ 
is a simple module over $\rats G$.  

Let us emphasize that by definition, QI-groups are transitive.
Thus the fixed space $\Fix(G)$ is generated by the 
all ones vector $(1,1,\dotsc,1)^t$,  
and so $\dim \Fix(G)= 1$.

\begin{thm}\label{t:QI_finitelycorep}
  Let $G \leq \Sym{d}$ be a QI-group.
  Then there is a constant $M$ depending only on the group $G$
    such that every core point is normalizer equivalent
    to a core point $w$ with $\norm{w}^2\leq M$.
  In particular, there are only finitely many core points for $G$ 
  up to normalizer equivalence.
\end{thm}

We divide the proof of \cref{t:QI_finitelycorep}
into a number of lemmas.
The idea is the following: 
We show that for any vector $z\in \ints^d$ there is 
some $c\in \C_{\GL(d,\ints)}(G)$ such that the projections
of $cz$ to the different irreducible real subspaces of 
$\Fix(G)^{\perp}$
have approximately the same norm.
(At the same time, this point $cz$ is one with minimal
norm in the orbit $\C_{\GL(\Lambda)}(G)z$.)
When $z$ is a core point,
at least one of these norms must be ``small''
by a fundamental result of Herr, Rehn, and 
Schürmann~\cite[Theorem~9]{herrrehnschue15}
(\cref{t:projectionbounded} below).

We begin with a short reminder of some character theory.
The facts we need can be found in any basic text on
representations of finite groups,
for example Serre's text~\cite{SerreLRFG}.
Saying that a group $G$ acts linearly on a 
(finite-dimensional) vector space~$V$ over some field~$K$
is equivalent to having a representation
$R\colon G\to \GL(V)$
(or even $R\colon G\to \GL(d,K)$ when $V=K^d $).
The \emph{character} $\chi$ of $R$ (or $V$) is the function
defined by $\chi(g) = \tr (R(g))$.
An \emph{irreducible character} is the trace of an irreducible
representation $R\colon G\to \GL(d,\compl)$ over 
the field of complex numbers~$\compl$.
The set of irreducible characters of
the group $G$ (over the complex numbers)
is denoted by $\Irr(G)$.
For finite groups $G$, this is a finite set.
Indeed, by the orthogonality relations,
the set $\Irr(G)$ is orthonormal with respect
to a certain inner product on the space of all
functions $G\to \compl$~\cite[Section~2.3, Thm.~3]{SerreLRFG}.

Every character of a finite group can be written uniquely as a 
nonnegative integer linear combination of irreducible 
characters.
This corresponds to the fact that for 
each representation $G\to \GL(V)$ on some vector space $V$ over 
$\compl$, we can write $V$
as a direct sum of irreducible, $G$-invariant 
subspaces \cite[\S 1.4, Thm.~2, \S 2.3, Thm.~4]{SerreLRFG}.

Suppose $\chi$ is the character of some representation
$R$ of the finite group $G$.
Then the eigenvalues of $R(g)$, where $g\in G$, must be
$\card{G}$th roots of unity.
Thus the values of $\chi$ are contained in the field
generated by the $\card{G}$th roots of unity.
We write $\rats(\chi)$ for the field generated by all values 
of $\chi$.
It follows that $\rats(\chi)$ is a finite Galois extension 
of $\rats$,
with abelian Galois group $\Gal(\rats(\chi)/\rats)$.

The following lemma appears in 
Dixon's paper~\cite[Lemma~6(b)]{dixon05}.

\begin{lemm}[Dixon~\cite{dixon05}]\label{l:decomp_ratirr}
  Let $G$ be a QI-group and let
  $\pi$ be the character of the corresponding permutation
  representation of $G$.
  Let $\chi\in \Irr G$ be an irreducible constituent of
  $\pi - 1$ (the character of $G$ on $\Fix(G)^{\perp}$).
  Then 
  \[ \pi = 1 + \sum_{ \alpha \in \Gamma } \chi^{\alpha},
     \quad\text{where}\quad 
     \Gamma = \Gal(\rats(\chi)/\rats).
  \]
\end{lemm}

For the moment, we work with the complex space
$\compl^d$, on which $G$ acts by permuting coordinates.
Recall that to each $\chi\in \Irr G$ there corresponds a central primitive
idempotent of the group algebra $\compl G$, namely
\[  e_{\chi} = \frac{\chi(1)}{\card{G}} \sum_{g\in G} \chi(g^{-1})g
   \in \Z(\compl G).
\] 
If $V$ is any $\compl G$-module, then
$e_{\chi}$ acts on $V$ as the projection onto
its $\chi$-homogeneous component.
So the image $e_{\chi}(V)$ coincides
with the set $\{v\in V\mid e_{\chi}v=v \}$, 
and the character of $e_{\chi}(V)$ is an integer multiple
of $\chi$~\cite[\S 2.6]{SerreLRFG}. 
In the present situation, it follows
from \cref{l:decomp_ratirr} that
\[ U:= e_{\chi} (\compl^d )
     = \{ v\in \compl^d \mid 
          e_{\chi}v=v
       \}
\] 
is itself an irreducible module affording the character $\chi$.
The projection $e_{\chi}$ maps the standard basis of $\compl^d$ 
to vectors contained in
$K^d$, where $K:= \rats(\chi)$.
Thus $U$ has a basis contained in $K^d$.
(This means that the representation corresponding to the linear action
of $G$ on $U$ can be described by matrices 
with all entries in $K$.
Thus $\chi$ 
is the character of a representation where all matrices have
entries in $K=\rats(\chi)$.)

Another consequence of \cref{l:decomp_ratirr} is 
that we have the decomposition
\[ \compl^d = \Fix(G) \oplus 
                      \bigoplus_{\gamma\in \Gamma}
                             U^{\gamma}.
\]
Here $U^{\gamma}$ means this:
Since $U$ has a basis in $\rats(\chi)^d$,
we can apply $\gamma$ to the coordinates of the 
vectors in such a basis.
The linear span of the result is denoted by $U^{\gamma}$.
This is independent of the chosen basis.

\begin{lemm}\label{l:endohom}
  Set
  $ A:= \C_{\mat_d(\rats)}(G)
   = \{a\in \mat_d(\rats) \mid \forall g\in G \colon
            ag=ga\}$, the full centralizer of $G$
  in the ring of $d\times d$-matrices over $\rats$.
  There is an algebra homomorphism
  $\lambda\colon A \to \rats(\chi)$
  such that each $a\in A$ acts on $U^{\gamma}$
  by multiplication with $\lambda(a)^{\gamma}$,
  and such that $\lambda(a^t) = \overline{\lambda(a)}$.
  There is another homomorphism
  $m\colon A \to \rats$ such that
  \[ A \iso \rats \times \rats(\chi)
    \quad \text{via}\quad
    a \mapsto (m(a),\lambda(a)).     
  \]
\end{lemm}

The isomorphism $A\iso \rats \times \rats(\chi)$
appears in Dixon's paper~\cite[Lemma~6(d)]{dixon05}
and follows from \cref{l:decomp_ratirr}
together with general results in representation theory.
But as we need the specific properties of the map 
$\lambda$ from the lemma, we give a detailed proof.

\begin{proof}[Proof of \cref{l:endohom}]
  Suppose the matrix $a$ centralizes $G$,
  and let $\lambda(a)\in \compl$ be an eigenvalue
  of $a$ on $U$. 
  The corresponding eigenspace is $G$-invariant
  since $a$ centralizes $G$.
  Since $U$ is irreducible, $U$ is contained
  in the eigenspace of $\lambda(a)$. 

  When $a\in A \subseteq \mat_d(\rats)$, 
  then $a$ maps 
  $U \cap \rats(\chi)^d\neq \{0\}$ to itself, and thus
  $\lambda(a)\in \rats(\chi)$.
  This defines the algebra homomorphism
  $\lambda\colon A \to \rats(\chi)$.
  
  When $u\in U \cap \rats(\chi)^d$, $\gamma\in \Gamma$,
   and $a\in A$, then
  $a u^{\gamma} = (au)^{\gamma} = \lambda(a)^{\gamma}u^{\gamma}$.
  Thus $a$ acts as $\lambda(a)^{\gamma}$ on $U^{\gamma}$.

  Each $a\in A$ acts also on the one-dimensional 
  fixed space by multiplication with
  some $m(a)\in \rats$.
  As
  \[ \compl^d = \Fix(G) \oplus 
                        \bigoplus_{\gamma\in \Gamma}
                               U^{\gamma},
  \]
  we see that the space $\compl^d$ has a basis of
  common eigenvectors for all $a\in A$.
  With respect to this basis, each $a$ is a diagonal matrix,
  where $m(a)$ appears once and $\lambda(a)^{\gamma}$
  appears $\chi(1)$-times for each $\gamma\in \Gamma$.
  In particular,
  the map $A\ni a \mapsto (m(a), \lambda(a))$
  is injective.
  
  Since $G$ acts orthogonally with respect to the standard 
  inner product on $\compl^d$,
  the above decomposition into irreducible subspaces
  is orthogonal and we can find an orthonormal basis
  of common eigenvectors of all $a\in A$.
  From this, it is clear that
  $\lambda(a^t) = \lambda(a^{*}) = \overline{\lambda(a)}$.

  To see that $a\mapsto (m(a),\lambda(a))$ is onto,
  let $(q, \mu)\in \rats \times \rats(\chi)$.
  Define
  \begin{align*}
     \phi(q,\mu)
       &:= qe_1 +
           \sum_{\gamma\in \Gamma} (\mu e_{\chi})^{\gamma}
       \\
       &= q\frac{1}{\card{G}} \sum_{g\in G} g
          + \frac{\chi(1)}{\card{G}}
              \sum_{g\in G} 
              \left( \sum_{\gamma\in \Gamma}
                       \big( \mu \chi(g^{-1})
                       \big)^{\gamma}
              \right) g                       
      \\
      &\in \Z(\rats G).
  \end{align*}
  Then the corresponding map
  $v \mapsto \phi(q,\mu)v$ is in $A$,
  and from 
  $\phi(q,\mu)e_1 = qe_1$ and
  $\phi(q,\mu)e_{\chi} = \mu e_{\chi}$
  we see that $m(\phi(q,\mu))=q$ and
  $\lambda(\phi(q,\mu)) = \mu$.  
  This finishes the proof that $A\iso \rats \times \rats(\chi)$.
\end{proof}

\begin{lemm}\label{l:realdecomp}
  Set $W:= (U + \overline{U})\cap \reals^d$.
  Then the decomposition of\/ $\reals^d$ into irreducible
  $\reals G$-modules is given by
  \[ \reals^d = \Fix(G) 
                   \oplus
                   \bigoplus_{\alpha\in \Gamma_0}
                      W^{\alpha}
      , \quad
      \Gamma_0 = \Gal((\rats(\chi)\cap \reals)/\rats).
  \]
  (In particular, $W$ is irreducible as an $\reals G$-module.)
  For $w\in W^{\alpha}$ and $a\in A$, we have
  $\norm{aw}^2 
   = \left( \overline{\lambda(a)} \lambda(a) 
     \right)^{\alpha} 
     \norm{w}^2$.    
\end{lemm}

\begin{proof}
  When $\rats(\chi)\subseteq \reals$, then
  $ \overline{U} = U$ and 
  $W = U\cap \reals^d$.
  The result is clear in this case.
  
  Otherwise, we have $ U\cap \reals^d =\{0\} $
  and $U\cap \overline{U} = \{0\}$,
  and so $W = (U \oplus \overline{U})\cap \reals^d\neq \{0\}$,
  and thus again $W$ is simple over $\reals G$.
  
  The extension $\rats(\chi)/\rats $ has an abelian Galois group,
  and thus $\rats(\chi)\cap \reals$ is also Galois over~$\rats$.
  The Galois group $\Gamma_0$ is isomorphic to the
  factor group $\Gamma / \{\id, \kappa\}$,
  where $\kappa$ denotes complex conjugation.
  Suppose $\alpha \in \Gamma_0$ is the restriction
  of $\gamma \in \Gamma$ to $\rats(\chi)\cap \reals$.
  Then
  \[ W^{\alpha} 
      = \left( (U+\overline{U})\cap \reals^d \right)^{\alpha}
      = (U^{\gamma} + \overline{U}^{\gamma}) \cap \reals^d
     = (U^{\gamma} + U^{\kappa\gamma})\cap \reals^d.
  \]
  The statement about the decomposition follows.
  
  The last statement is immediate from \cref{l:endohom}. 
\end{proof}

\begin{lemm}\label{l:full_latt}
  Let $C:= \C_{\GL(d,\ints)}(G)$, and define
  \[ L\colon C \to \reals^{\Gamma_0}
     , \quad
     L(c) := \big( \log (\overline{\lambda(c)} \lambda(c))^{\alpha}  
             \big)_{\alpha\in \Gamma_0}.
  \]
  Then the image $L(C)$ of $C$ under this map is a full lattice
  in the hyperplane
  \[ H = \bigg\{ (x_{\alpha})_{\alpha\in \Gamma_0}
             \mid \sum_{\alpha\in \Gamma_0} x_{\alpha} = 0
         \bigg\}.
  \]
\end{lemm}

We will derive this lemma from the following 
version of Dirichlet's unit theorem~\cite[Satz~I.7.3]{Neukirch07_AZT}:

\begin{lemm}\label{l:diri_lattice}
  Let $K$ be a finite field extension over $\rats$,
  let $\alpha_1$, $\dots$, $\alpha_r \colon K\to \reals$
  be the different real field embeddings of $K$, 
  and let
  $\beta_1$, $\overline{\beta_1}$, $\dots$, 
  $\beta_s$, $\overline{\beta_s} \colon K\to \compl$
  be the different complex embeddings of $K$,
  whose image is not contained in $\reals$.
  Let $O_K$ be the ring of algebraic integers in $K$
  and $l\colon K^{*} \to \reals^{r+s}$ the map 
  \[ z \mapsto 
     l(z)=
     (\log \abs{ z^{\alpha_1} }, \dotsc,
      \log \abs{ z^{\alpha_r} },
      \log \abs{ z^{\beta_1} }, \dotsc,
      \log \abs{ z^{\beta_s} }
     ) \in \reals^{r+s}.
  \] 
  Then the image $l(\Units(O_K))$ 
  of the unit group of $O_K$ under $l$ is a full lattice in the
  hyperplane
  \[  H = 
      \bigg\{ x\in \reals^{r+s}
         \mid
         \sum_{i=1}^{r+s} x_i = 0
      \bigg\}.
  \]
\end{lemm}

In the proof of \cref{l:full_latt},
we will apply this result to 
$K = \rats(\chi)$.
Set $F= K\cap \reals$, 
$\Gamma_0 = \Gal(F/\rats)$
and $\Gamma = \Gal(K/\rats)$.
If $F=K\subseteq \reals$, then
$r= \card{K:\rats}$ and $s=0$.
In this case, 
$\{\alpha_1, \dotsc, \alpha_r\} = \Gamma = \Gamma_0$.
If $K \not\subseteq\reals$, 
then $\card{K:F}=2$,
$r=0$, and $s= \card{F:\rats}$.
In this case, we may identify the set
$\{\beta_1,\dotsc, \beta_s\}$ with
the Galois group $\Gamma_0$: 
for each $\alpha \in \Gamma_0$, there are two
extensions of $\alpha$ to the field $K$,
and these are complex conjugates of each other.
Thus we get a set $\{\beta_1,\dotsc, \beta_s\}$
as in \cref{l:diri_lattice} by choosing 
exactly one extension for each $\alpha\in \Gamma$.
The map $l$ is independent of this choice anyway.

It follows that in both cases, we may rewrite the map $l$
(somewhat imprecisely) as
\[ l(z) = \big( \log \abs{z^{\alpha}} \big)_{\alpha\in \Gamma_0}.
\]

\begin{proof}[Proof of \cref{l:full_latt}]
First notice that the entries of $L(c)$ 
can be written as
\begin{align*}
   \log\big(\overline{\lambda(c)} \lambda(c) 
       \big)^{\alpha}
      = \log \big( \overline{\lambda(c)^{\alpha}} 
                   \lambda(c)^{\alpha}
             \big)
      = \log \abs{ \lambda(c)^{\alpha}}^2
       &= 2 \log \abs{\lambda(c)^{\alpha}},
\end{align*} 
where we tacitly replaced $\alpha$ by an extension to $\rats(\chi)$
when $\rats(\chi)\not\subseteq \reals$. 
Thus  $L(c) = 2 l (\lambda(c))$
for all $c\in C$,
with $l$ as in \cref{l:diri_lattice}.

In view of \cref{l:diri_lattice},
it remains to show that 
the group $\lambda(C)$ has finite index in
$\Units(O_K)$.
We know that $C$ is the group of units in
$\C_{\mat_d(\ints)}(G) \iso \enmo_{\ints G}(\ints^d)$, 
which is an order in
$A \iso \rats \times K$.
Another order in $\rats \times K$ 
(in fact, the unique maximal order) is
$\ints \times O_K$ with unit group
$\{\pm 1\} \times \Units(O_K)$.
By \cref{l:suborder},
it follows that $C$ has finite index in 
$\{\pm 1\} \times \Units(O_K)$.
Thus $\lambda(C)$ has finite index in 
$\Units(O_K)$ and the result follows.
\end{proof}

  For each $v\in \reals^d$, let 
  $v_{\alpha}$ be the orthogonal projection
  of $v$ onto the simple subspace $W^{\alpha}$.

\begin{lemm}\label{l:almostequal}
  There is a constant $D$, depending only on the group $G$,
  such that for every $v\in \reals^d $ with 
  $v_{\alpha}\neq 0$ for all $\alpha\in \Gamma_0$,
  there is a
  $c\in C$ with
  \[ \frac{\norm{(cv)_{\alpha}}^2}{\norm{(cv)_{\beta}}^2}
     \leq D
  \]
  for all $\alpha$, $\beta \in \Gamma_0$.
\end{lemm}

As $\Fix(G)^{\perp}\cap \rats^d$ is a simple module,
the assumption $v_{\alpha}\neq 0$ for all $\alpha$
holds in particular for all 
$v \in \rats^d \setminus \Fix(G)$.

\begin{proof}[Proof of \cref{l:almostequal}]
  By \cref{l:full_latt}, there is a compact set~$T$,
  \[ T \subset H = \bigg\{ (x_{\alpha}) \in \reals^{\Gamma_0}
                       \mid
                       \sum_{\alpha} x_{\alpha} = 0
                   \bigg\},
  \]
  such that $H = T + L(C)$.
  (For example, we can choose $T$ as a fundamental 
   parallelepiped of the full lattice $L(C)$ in $H$.)
  
  For $v\in \reals^d$ as in the statement of the lemma, define
  \[ N(v) = \big( \log \norm{v_{\alpha}}^2 \big)_{\alpha}
            \in \reals^{\Gamma_0}.
  \]
  Let $S\in \reals^{\Gamma_0}$ be the vector having 
  all entries equal to
  \[ s := \frac{1}{\card{\Gamma_0}} \sum_{\alpha} \log\norm{v_{\alpha}}^2.
  \]
  This $s$ is chosen such that
  $N(v) - S \in H$.
  Thus there is $c\in C$ such that
  $L(c) + N(v) - S  \in T$, say
  $L(c) + N(v) - S = t = (t_{\alpha})$.
  
  As 
  \[ \norm{ (cv)_{\alpha}}^2
     = \norm{ c v_{\alpha}}^2
     = \left( \overline{\lambda(c)} \lambda(c)
       \right)^{\alpha} \norm{ v_{\alpha}}^2,
  \]
  it follows that
    \[  N(cv) = L(c) + N(v)
    \]   
  in general.
  Thus
  \begin{align*}
    \log\norm{cv_{\alpha}}^2 - \log\norm{cv_{\beta}}^2
      &=     (\log\norm{cv_{\alpha}}^2 -s) 
            - (\log\norm{cv_{\beta}}^2 - s)
      \\
      &= (N(cv) - S)_{\alpha} - (N(cv)-S)_{\beta}
      \\
      &= t_{\alpha} - t_{\beta}
      \\
      & \leq 
        \max_{\alpha, t} t_{\alpha} - \min_{\beta, t} t_{\beta}
        =: D_0.       
  \end{align*}
  This maximum and minimum exist
   since $T$ is compact.
  The number $D_0$ may depend on the choice of the set $T$, 
  but not on $v$ or $c$.
  Thus $\norm{cv_{\alpha}}^2/\norm{cv_{\beta}}^2$ is bounded
  by $ D := e^{D_0} $. 
\end{proof}

We see from the proof that we get a bound whenever we have
a subgroup $C_0$ of $\C_{\GL(d,\ints)}(G)$ such that
$L(C_0)$ is a full lattice in the hyperplane $H$.
Of course, we do not get the optimal bound then,
but in practice it may be difficult to compute the full
centralizer.

We will prove \cref{t:QI_finitelycorep}
by combining the last lemma with the following fundamental 
result~\cite[Theorem~9]{herrrehnschue15}
(which is actually true for arbitrary matrix
 groups~\cite[Theorem~3.13]{rehn13diss}).

\begin{thm}\label{t:projectionbounded}
  Let $G\leq \Sym{d}$ be a transitive permutation group.
  Then there is a constant $C$ (depending only on $d$)
  such that for each core point~$z$,
  there is a non-zero invariant subspace $U\leq \Fix(G)^{\perp}$
  over $\reals$
  such that $\norm{z|_{U}}^2 \leq C$.
\end{thm}

In our situation, the $W^{\alpha}$ from \cref{l:realdecomp}
are the only irreducible subspaces, and thus
for every core point $z$ there is some
$\alpha\in \Gamma_0$
with $\norm{z_{\alpha}}^2 \leq C$.

\begin{proof}[Proof of \cref{t:QI_finitelycorep}]
  Let $z$ be a core point with $z\notin \Fix (G) $.
  We want to show that there is 
    a $c\in \C_{\GL(d,\ints)}(G)$ and a vector 
    $b\in \Fix(G)\cap \ints^d$,
    such that 
    $ \norm{cz +b} \leq M$,
  where $M$ is a constant depending only on $G$ and not on~$z$.
  By \cref{l:almostequal}, there is
  $c \in \C_{\GL(d,\ints)}(G)$ such that 
  $\norm{cz_{\alpha}}^2 \leq D \norm{cz_{\beta}}^2$
  for all $\alpha$, $\beta\in \Gamma$,
  where $D$ is some constant depending only on $G$
  and not on $z$.
  
  Since $y = cz$ is also a core point
  (\cref{l:equiv_core_pts}),
  \cref{t:projectionbounded} yields that
  there is a $\beta\in \Gamma$ with 
  $ \norm{ y_{\beta} }^2 \leq C$
  (where, again, the constant~$C$ depends only on the group,
  not on $z$).
  It follows that 
  the squared norms of the other projections $y_{\alpha}$ are bounded by 
  $CD$.
  Thus
  \[ \norm{ y|_{\Fix(G)^{\perp}} }^2
     \leq C + (\card{\Gamma}-1)CD
  \]
  is bounded.

  Since the projection to the fixed space can be bounded
  by translating with some $b\in \Fix(G)\cap \ints^d  $,
  the theorem follows.
\end{proof}

\begin{example}\label{ex:c5}
  Let $p$ be a prime,
  and let $G = C_p \leq \Sym{p}$ be generated by a $p$-cycle
  acting on $\reals^p$ by (cyclically) permuting coordinates.
  Then $G$ is a QI-group.
  (Of course, every transitive group of prime degree
   is a QI-group.)
  For $p$ odd,
  $\reals^p$ decomposes into $\Fix(G)$ and 
  $(p-1)/2$ irreducible subspaces of dimension $2$.
  Here the lattice can be identified with the group ring
  $\ints G$, and thus
  $\C_{\GL(p,\ints)}(G) \iso \Units(\ints G)$.
  The torsion free part of this unit group is a free abelian
  group of rank $(p-3)/2$.

  Let us see what constant we can derive for 
  $p=5$.
  For concreteness, let $g=(1,2,3,4,5)$
  and $G= \erz{g}$.
  We have the decomposition
  \[ \reals^5 = \Fix(G) \oplus W \oplus W'.
  \]
  The projections from $\reals^5$ onto $W$ and $W'$
  are given by
  \begin{align*} 
    e_W &= \frac{1}{5}(2+ag+bg^2+bg^3+ag^4)
    ,&
    a &= \frac{-1+\sqrt{5}}{2},
     \\
     e_{W'} &=  \frac{1}{5}(2+bg+ag^2+ag^3+bg^4)
     ,&
     b &= \frac{-1-\sqrt{5}}{2}.
  \end{align*}
  The centralizer of $G$ has the form
  \[ \C_{\GL(5,\ints)}(G)
     = \{ \pm I\} \times G \times \erz{ u },
  \]
   where $u$ is a unit of infinite order.
   Here we can choose
    $u= -1+g+g^4$ with inverse $-1+g^2+g^3$.
   To $u$ corresponds the matrix
   \begin{equation} \label{eqn:CentralizerMatrix}
      \begin{pmatrix}
        -1 & 1 & 0 & 0 & 1 \\
        1 & -1 & 1 & 0 & 0 \\
        0 & 1 & -1 & 1 & 0 \\
        0 & 0 & 1 & -1 & 1 \\
        1 & 0 & 0 & 1 & -1
      \end{pmatrix} \in \GL(5,\ints)
   .
   \end{equation}
   This unit acts on $W$ as  $-1+a$ and on $W'$
   as $-1+b$.
   For the constant~$D$ of \cref{l:almostequal},
   we get $D=(b-1)^2 = 2-3b = (7+3\sqrt{5})/2$.
   For the constant $C$ in \cref{t:projectionbounded}, 
   we get a bound $C = 48/5$ (from the proof).
   We can conclude that every core point is equivalent to one
   with squared norm smaller
   than $M= (2/5) + (48/5)(1 + 2-3b) \approx 50.6$.
   
   We can get somewhat better bounds by applying 
   \cref{t:projectionbounded} ``layerwise''.
   The $k$-layer is, by definition, the set of all
   $z\in \ints^d$ with $\sum z_i = k$.
   In our example, every lattice point is equivalent to one in 
   layer $1$ or layer $2$.
 
   For example, it can be shown that 
   each core point in the $1$-layer 
   is equivalent to a point $z$
   with $\norm{z}^2 \leq 31$.
   However, this bound is still far from optimal.
   Using the computer algebra system \texttt{GAP}~\cite{GAP486},
   we found that the only core points of
   $C_5$ in the $1$-layer up to normalizer equivalence are just
   \begin{align*}
     & (1,0,0,0,0)^t, && (1,1,0,0,-1)^t, &&(1,1,1,0,-2)^t,
     \\
      & (2,1,0,-1,-1)^t, && (2,1,-2,0,0)^t.
   \end{align*}
   (The normalizer $\N_{\GL(5,\ints)}(G)$ is generated by
    the centralizer and the permutation matrix corresponding to
    the permutation $(2,3,5,4)$.)
   For completeness, we also give a list of core points 
   up to normalizer equivalence in the  $2$-layer:
   \begin{align*}
     & (1,1,0,0,0)^t, && (1,1,1,0,-1)^t, && (2,1,0,0,-1)^t,
     \\
     & (2,1,1,-1,-1)^t, && (2,1,1,-2,0)^t.
   \end{align*}
   Every nontrivial core point for $C_5$ is normalizer equivalent 
   to exactly  one of these ten core points.
   
   For this example, 
   an infinite series of core points of the form
       \[ (f_{j+1},0,f_j,f_j,0)^t,
       \]
   where $f_j$ is the $j$th Fibonacci number,
   was found by Rehn~\cite[5.2.2]{rehn13diss}.
   Each point in this series is normalizer equivalent 
   to one of the two obvious core points
    $(1,0,0,0,0)^t$ and $(1,0,1,1,0)^t$.
   This follows from
    \[ (1-g-g^4)(f_{j+1},0,f_j,f_j,0)^t
       =(f_{j+1},-f_{j+2},0,0,-f_{j+2})^t
    \]
    and thus
    \begin{equation*}
     (1-g-g^4)(f_{j+1},0,f_j,f_j,0)^t 
       + f_{j+2}(1,1,1,1,1)^t
        = (f_{j+3},0,f_{j+2},f_{j+2},0)^t.    
    \end{equation*}
\end{example}

\begin{example}
  Now set
  \[ G = \erz{ (1,2,3,4,5), \; (1,4)(2,3) }
      \iso D_5,
  \]
  the dihedral group of order $10$.
  Then
  \[ C_{\GL(5,\ints)}(G) = \{ \pm I\} \erz{u},
  \]  
  where $u$ is as in the previous example. 
  The normalizer of $G$ is the same as that of the cyclic group
  $C_5=\erz{(1,2,3,4,5)}$.
  In particular, normalizer equivalence for $D_5$ and $C_5$
  is the same equivalence relation.
  Of the core points from the last example,
  only $(1,0,0,0,0)^t$ and $(1,1,0,0,0)^t$
  are also core points for $D_5$.
  (In fact, for most of the other points, we have
   some lattice point on an interval between 
   two vertices---for example
   $(1,0,0,0,0)^t = (1/2) 
   \big((1,1,0,0,-1)^t + (2,5)(3,4)(1,1,0,0,-1)^t \big)$.
  Thus there are only two core points up to normalizer equivalence
  in this example.  
\end{example}

\begin{remark}
  The number of core points up to normalizer equivalence
  seems to grow quickly for cyclic groups of prime order.
  For $p=7$, we get $515$ core points up to normalizer
  equivalence.
\end{remark} 

Herr, Rehn, and Schürmann~\cite{herrrehnschue15}
conjectured that a finite transitive
permutation group $G$ has infinitely many core points
up to translation equivalence if the group is not 
$2$-homogeneous. 
This conjecture is still open
but is known to be true in a number of special cases,
including imprimitive permutation groups and 
all groups of degree $d\leq 127$.

It is known that a permutation group 
$G \leq \Sym{d}$ is $2$-homogeneous
if and only if $\Fix_{\reals^d}(G)^{\perp}$
is irreducible~\cite[Lemma~2(iii)]{cameron72}.
In this case, there are only finitely many core points up to
translation equivalence~\cite[Corollary~10]{herrrehnschue15}.

We propose the following conjecture,
which is the converse of \cref{t:QI_finitelycorep}:

\begin{conjecture}
  Let $G\leq \Sym{d}$ be a transitive permutation 
  group such that 
  $\Fix(G)^{\perp}$ contains a rational
  $G$-invariant subspace other than
  $\{0\}$ and $\Fix(G)^{\perp}$ itself.
  Then there are infinitely many core points up to
  normalizer equivalence.
\end{conjecture}

This can be seen as a generalization
of the Herr-Rehn-Schürmann conjecture,
since translation equivalence refines normalizer equivalence,
and since whenever
$\Fix(G)^{\perp}$ contains a nontrivial irrational $G$-invariant
subspace,  there are infinitely many core points up to translation
equivalence by \cref{t:finitecrit}
(or~\cite[Theorem~32]{herrrehnschue15}).

\section{Application to integer linear optimization} 
\label{sec:app}

In this last section we describe a possible application of 
the concept of normalizer equivalence
to symmetric integer linear optimization problems.
For many years it has been known that symmetry
often leads to difficult problem instances in integer optimization.
Standard approaches like branching usually work
particularly poorly when large symmetry groups are present,
since a lot of equivalent subproblems have to be dealt with in such cases.
Therefore, in recent years several new methods for exploiting symmetries 
in integer linear programming have been developed.
See, for example,
\cite{Margot03,Friedman07,BulutogluMargot08,%
KaibelPfetsch08,LinderothMT09,OstrowskiLRS11,
GhoniemSherali11,FischettiLiberti12,HojnyPfetsch17}
and the surveys by Margot \cite{Margot2009} and 
Pfetsch and Rehn \cite{PfetschRehn2017}
for an overview. 
These methods (with the exception of \cite{FischettiLiberti12}) fall broadly into two classes:
Either they modify the standard branching approach,
using isomorphism tests or isomorphism free generation to
avoid solving equivalent subproblems, or they use techniques to
cut down the original symmetric problem to a less symmetric one,
which contains at least one element of each orbit of solutions.
By now, many of the
leading commercial solvers, like
\texttt{CPLEX} \cite{cplex}, \texttt{Gurobi} \cite{gurobi}, 
and \texttt{XPRESS} \cite{xpress},
have included some techniques
to detect and exploit special types of symmetries.
Accompanying their computational survey \cite{PfetschRehn2017}, 
Pfetsch and Rehn also published implementations of some
symmetry exploiting algorithms for \texttt{SCIP} \cite{scip},
like isomorphism pruning and orbital branching.

Core points were introduced as an additional 
tool to deal with symmetries in integer convex optimization problems.
Knowing the core points for a given symmetry group
allows one to restrict the search for optima 
to this subset of the integer vectors \cite{herrrehnschue13}.
There are many possible ways how core points could be used.
For instance,  one could use the fact that 
core points are close to invariant subspaces, 
by adding additional quadratic constraints 
(second order cone constraints). 
In the case of QI-groups, 
hence with finitely many core points 
up to normalizer equivalence 
(\cref{t:QI_finitelycorep}),
one could try to systematically run through   
core points satisfying the problem constraints.

In contrast to the aforementioned approaches,
we here propose natural reformulations of symmetric
integer optimization problems using the normalizer of the symmetry
group.
Recall that a general standard form of an integer linear optimization
problem is
\begin{equation} \label{eqn:standardILP}
    \max c^t x 
    \quad \text{such that} \quad 
    A x \leq b,\; 
    x\in \ints^d,
\end{equation}
for some given matrix $A$ and vectors $b$ and $c$, 
all of them usually rational.
If $c=0$, then
we have a so-called 
\emph{feasibility problem}, 
asking simply whether or not
there is an integral solution to a 
given system of linear inequalities.
Geometrically, we are asking
whether some polyhedral set (a polytope, if bounded) 
contains an integral point.

A group $G\leq \GL(d,\ints)$ is called a 
\emph{group of symmetries} 
of problem~\eqref{eqn:standardILP}
if the constraints $Ax\leq b$ 
and the linear objective function~$c^tx$ are
invariant under the action of $G$ on $\reals^d$,
that is, if  $c^t (gx) = c^t x$ and $A (gx) \leq b$ 
for all $g\in G$ whenever $A x \leq b$.
The first condition is, for instance, satisfied 
if $c$ is in the fixed space $\Fix(G)$.
Practically, computing a group of symmetries for a given problem is
usually reduced to the problem of finding symmetries 
of a suitable colored graph~\cite{BremnerETAL2014,PfetschRehn2017}.
Quite often in optimization, attention is restricted to groups 
$G \leq \Sym{d}$ acting on $\reals^d$ by permuting coordinates.

Generally, a linear reformulation of a problem as 
in~\eqref{eqn:standardILP} can be obtained 
by an integral linear substitution $x\mapsto Sx$ 
for some matrix~$S \in \GL(d,\ints)$:
\begin{equation} \label{eqn:reformulatedILP}
    \max (c^t S) x 
    \quad \text{such that} \quad 
    (AS) x \leq b,\: x\in \ints^d.
\end{equation}
(More generally, one can use integral 
affine substitutions $x\mapsto Sx + t$
with $S\in \GL(d,\ints)$ and $t\in \ints^d$.
For simplicity, we assume $t=0$ in the discussion to follow.)
We remark that reformulations as in 
\eqref{eqn:reformulatedILP} with a
matrix $S \in \GL(d,\ints)$ can of course be applied to any linear
integer optimization problem. 
In fact, this is a key idea of 
Lenstra's famous polynomial time algorithm 
in fixed dimension $d$~\cite{Lenstra1983}.
In Lenstra's algorithm, the transformation matrix $S$ is chosen 
to correspond to a suitable LLL-reduction of the lattice,
such that the transformed polyhedral set 
$\{ x \in \reals^d \mid (AS) x \leq b\}$ is sufficiently round.
This idea has successfully been used for different problem classes
of integer linear optimization problems 
(for an overview see~\cite{AardalEisenbrand2005}).
The main difficulty is the choice of an appropriate unimodular matrix
$S$ which simplifies the optimization problem.

If the symmetry group of an optimization problem 
contains the group $G$, 
then it is natural to choose matrices $S$ 
which keep the problem $G$-invariant.
When $S$ is an element of the normalizer 
$\N_{\GL(d,\ints)}(G)$,
problem~\eqref{eqn:standardILP} is 
$G$-invariant if and only if \eqref{eqn:reformulatedILP} is
$G$-invariant.
Note also that then $(c^t S)^t$ is in $\Fix (G)$.

We illustrate the idea
with a small concrete problem instance 
of~\eqref{eqn:standardILP} 
which is invariant under the cyclic group $C_5$.
In particular, using core points,
we construct 
$C_5$-invariant integral optimization problems 
that are quite hard or even impossible to solve for 
state-of-the-art commercial solvers 
like \texttt{CPLEX} or \texttt{GUROBI}. 
For instance, this is often the case when
the constraints $Ax\leq b$ can be satisfied by real vectors $x$, 
but not by integral ones. 

\begin{example}
The orbit polytope $P(C_5,z)$ of some integral point~$z$
has a description with linear inequalities of the form
$x_1 + \dotsb + x_5 = k$ and $Ax \leq b$,
where $A$ is a circulant $5\times 5$-matrix
\[
A=\begin{pmatrix}
a_1 & a_2 & a_3 & a_4 & a_5 \\
a_2 & a_3 & a_4 & a_5 & a_1 \\
a_3 & a_4 & a_5 & a_1 & a_2 \\
a_4 & a_5 & a_1 & a_2 & a_3 \\
a_5 & a_1 & a_2 & a_3 & a_4
\end{pmatrix}
\]
with integral entries $a_1$, $\dotsc $, $a_5$, 
and $b \in \ints^5$ satisfies $b_1=\dotsb = b_5$. 
If $z$ is a core point and if we replace
$b_i$ by $b_i':=b_i-1$, then we get a system of inequalities
having no integral solution.

Applying this construction to the core point
\[ z=U^{10}\cdot (1,1,1,0,-2)^t,
\]
where $U$ is the 
matrix from~\eqref{eqn:CentralizerMatrix} in \cref{ex:c5},
we get parameters
\begin{alignat*}{4}
    a_1 &= 515161, & \quad 
    a_2 &= 18376, & \quad 
    a_3 &= -503804, \\
    a_4 &= -329744, & \quad 
    a_5 &= 300011, & \quad
    b_1' &= 60.
\end{alignat*}
We can vary the values of $k\equiv 1 \mod 5$
(geometrically, this corresponds to translating the polytope
by some integral multiple of the all-ones vector).
This gives a series of problem instances
on which the commercial solvers very often not finish within 
a time limit of 10000~seconds on a usual desktop computer. 
For $k=1$, which seems computationally the easiest case,
a solution still always takes more than 4000~seconds.

However, knowing that a given problem such as the above is 
$C_5$-invariant, we can try
to find  an easier reformulation \eqref{eqn:reformulatedILP} 
by using matrices from the centralizer.
As a rule of thumb, we assume that a transformed problem with smaller
coefficients is ``easier.'' 
Here, the torsion free part of the centralizer is generated
by the matrix $U$ from~\eqref{eqn:CentralizerMatrix} in
\cref{ex:c5}, and so the only possible choices
for $S$ are $U$ and $U^{-1}$.
(A matrix of finite order will probably not simplify
a problem significantly.)
Here, applying $S=U$ yields an easier problem,
and one quickly finds that after applying $S$ ten times,
the problem is not simplified further by 
applying $U$ (or $U^{-1}$).
In other words, 
we transform the original problem instance with $U^{10}$. 
This yields an
equivalent $C_5$-invariant feasibility problem,
which is basically instantly solved by the commercial solvers
(finding that there is no integral solution).

As far as we know, this approach is in particular by far better 
than any previously known one that uses the 
symmetries of a cyclic group. 
One standard approach is, for example, to add symmetry-breaking
inequalities 
$x_1\leq x_2, \ldots, x_1\leq x_5$. 
This yields an improved performance in some cases
but is far from the order of computational gain
that is possible with our proposed reformulations.
\end{example}

In general, when an integer linear program~\eqref{eqn:standardILP} 
is invariant
under a QI-group $G$, and when it has any solutions at all,
then \cref{t:QI_finitelycorep} tells us that there is a 
transformation $x\mapsto Sx+ t$ with $S\in \N_{\GL(d,\ints)}(G)$
such that the reformulated problem has a feasible solution
in a given finite set 
(a set of representatives of core points under normalizer equivalence).
Heuristically, this means that we should be able to transform
any $G$-invariant problem into one of bounded difficulty:
By \cref{l:almostequal}, for any vector $x\in \reals^d$, 
there is an element $S\in \C_{\GL(d,\ints)}(G)$ 
such that the projections of $S x$ to the different 
$G$-invariant subspaces have approximately the same norm.
This means that the orbit polytope of $S x$ is ``round.''
 
Our approach is particularly straightforward 
when the torsion free part of 
the centralizer $\C_{\GL(d,\ints)}(G)$ has just rank~$1$,
as in the example with $G=C_5$ above.
When the centralizer contains a free abelian group of
some larger rank, then it is less clear 
how to reduce the problem efficiently. 
A possible heuristic is as follows:
Recall that in \cref{l:full_latt}, we described a map $L$
which maps the centralizer,
and thus its torsion-free part of rank~$r$ (say),
onto a certain lattice  in $\reals^{r+1}$.
This maps the problem of finding a
reformulation~\eqref{eqn:reformulatedILP}
with ``small'' $AS$ to a minimization problem
on a certain lattice. 
For example, when we minimize 
$\norm{AS}$, this translates to minimizing a convex function
on a lattice. 
So we can find a good reformulation by finding
a lattice point in $\reals^{r+1}$ which is close to the minimum,
using, for instance, LLL-reduction.
This will be further studied in a forthcoming paper.

\section*{Acknowledgments}
We would like to thank the anonymous referees for several valuable
comments.
We also gratefully acknowledge support by DFG grant SCHU 1503/6-1.


\printbibliography


\end{document}